\documentclass[final]{siamart250211}

\usepackage{amsmath,amssymb}
\usepackage{mathrsfs}        
\usepackage{bm}              
\usepackage{algorithm,algpseudocode}
\usepackage{enumerate,enumitem}   
\usepackage{xcolor}         
\usepackage{cite}

\numberwithin{theorem}{section}

\newsiamremark{remark}{Remark}

\newcommand{\Tr}{\operatorname{Tr}}             
\newcommand{\Hil}{\mathscr{H}}                  

\title{Local spectral approximation of unbounded operators: non-asymptotic and unified error quantification for subspace methods}

\author{
    Timothy Stroschein \\
    \small ETH Zurich, Department of Chemistry and Applied Biosciences, \\
    \small Vladimir-Prelog-Weg 2, 8093 Zurich, Switzerland\\
    \texttt{\small timothy.stroschein@phys.chem.ethz.ch}
}
\date{}

\begin{document}

\maketitle

\begin{abstract}
   We introduce a framework for subspace methods which approximate the spectra of self-adjoint, unbounded operators in a local region.
   Using the projection-valued measure, we derive integrated spectral inequalities that also apply to unbounded operators. 
   Our framework is non-asymptotic, gap-independent, and enables a unified error quantification of numerical routines subject to multiple error sources.
   Furthermore, we formalize the class of methods applicable to our framework, and establish a rigorous foundation for 
   \emph{dimension detection in the presence of noise} as solution to frequent numerical artifacts such as spectral pollution. 
   The practical relevance of this non-asymptotic analysis is substantiated by its recent application to sampled prolate filter diagonalization, 
   where it successfully predicted a sharp accuracy transition linking spectral density to the minimal observation time required to decompose a signal.
\end{abstract}

\section{Introduction}

Formulating operator problems on infinite-dimensional Hilbert spaces is foundational to the natural and information sciences~\cite{Reed1980MethodsOM, von_neumann_1927, dirac1930principles}.
However, the computation of quantitative predictions requires the reduction of this formalism to finite-dimensional matrix problems~\cite{Saad2011, Parlett1998}. 
The analytical description of this reduction is particularly difficult for unbounded operators, which are central to quantum mechanics~\cite{reed_iv_1978}.
In numerical practice, approximation schemes such as the Galerkin method suffer from spectral pollution and other artifacts which the theory struggles to explain ~\cite{boulton_spectral_2016, Boffi_2010,  CullumWilloughby1985}.
Therefore, rigorous error quantification in the conversion of infinite, unbounded operator problems into finite matrix problems remains a crucial challenge in computational physics~\cite{stroschein2025groundexcitedstateenergiesanalytic,stroschein2024prolatespheroidalwavefunctions}.

The difficulty lies in the limitations of existing approximation frameworks. 
Originating from functional analysis~\cite{Kato1995}, 
the theories of Chatelin and Babu\v{s}ka--Osborn focus on global operator convergence ($T_{n}\rightarrow T$)~\cite{Chatelin1983, BabuskaOsborn1991}. 
While these yields foundational results on convergence rates for operators with compact resolvent~\cite{BabuskaOsborn1989}, they prioritize abstract conceptual understanding over quantitative approximation. 
And in practice, the global approximation approach is particularly prone to spectral pollution~\cite{Boffi_2010, boulton_spectral_2016}.
In this work we diverge from global convergence analysis. Instead, we consider local spectral approximation of an unbounded, self-adjoint operator $H$ through a fixed trial subspace and seek non-asymptotic error bounds on the approximated eigenvalues.

The class of gap-dependent eigenpair bounds
associated with perturbation theory on large matrices is closer to this setting~\cite[Chapter 5]{StewartSun1990}\cite[Chapter 11]{Parlett1998}\cite[Chapter 4]{Saad2011}.
These bounds originate from quantifying the quality of a trial subspace through its principal angles with respect to the targeted spectral subspace,
as in the classical Davis-Kahn $\sin\Theta$ theorem~\cite{DavisKahan1970} and its recent refinements~\cite{Nakatsukasa2018}.
However, the estimates on the principal angles rely on a spectral gap separating the targeted eigenvalues from the remainder of the spectrum.
This \emph{a priori} assumption on the gap restricts the 
applicability of the theory. Additionally, while numerical stability indeed often relies on the separation of eigenvalues, the gap is not
always an appropriate measure of difficulty and accuracy~\cite{Parlett1996}.
However, the primary limitation of this literature is that it remains rooted in finite-dimensional matrix theory; the
error bounds typically depend explicitly on the operator norm, rendering them inapplicable to unbounded operators.\\

This work addresses the gap between these two approaches and 
provides a framework to describe all errors accumulated
as exact equations of functional analysis are mapped to finite matrix problems. We achieve this by introducing
an error measure defined through the projection
valued measure $P^{(H)}$ of the underlying operator $H$ and derive a sharp inequality on the approximate eigenvalues aligning with numerical intuition.
Our theory is non-asymptotic, gap-independent, and allows us to jointly estimate multiple error sources in addition to the subspace approximation.

To facilitate the transfer of analytical analysis to computational practice, we formalize the class of numerical methods compatible with our framework as \emph{subspace protocols}.
This allows us to identify \emph{accurate dimension detection in the presence of noise}~\cite{di_multiple_1985, WaxKailath1985, zhao_detection_1986, BaikSilverstein2006} as a critical step to ensure stable approximations and to avoid numerical artifacts such as spectral pollution and spurious eigenvalues \cite{Parlett1998, CullumWilloughby1985}.
We derive algebraic conditions guaranteeing accurate dimension detection and convergence.

Our theory is applicable to a wide range of computational fields relying on subspace expansions, as we remain agnostic to the specific mechanism generating the trial vectors. The practical relevance of this non-asymptotic framework was
recently substantiated in the analysis of sampled prolate filter diagonalization~\cite{stroschein2025groundexcitedstateenergiesanalytic, stroschein2024prolatespheroidalwavefunctions}. In that context, the framework successfully predicted a sharp accuracy transition linking spectral density to the necessary observation time for precise frequency resolution -- a result that asymptotic theories cannot capture.

The article is organized as follows. \Cref{sec:descr} formalizes the matrix eigenvalue problem that is the center of this work.
\Cref{sec:inequalities} presents the PVM-based spectral inequalities that bound the eigenvalue error with respect to the underlying operator $H$.
\Cref{sec:SubspaceProtocol_intro} introduces subspace protocols to address spectral pollution and establishes the stability conditions for dimension detection. Finally, \Cref{sec:outlook} concludes with an outlook on future applications.

\section{Problem description and formalism}
\label{sec:descr}

Let $H:\text{dom}H \to \mathscr{H}$ be a self-adjoint operator that is densely defined on a separable Hilbert space $\mathscr{H}$. 
Let $\mathcal{E} \subset \mathscr{H}$ be a spectral subspace of $H$ of dimension $m$.
We want to approximate the eigenvalues of $H$ in $\mathcal{E}$.

Assume a set of \emph{trial vectors} $v_1, \dots , v_m \in \mathcal{Q}(H)$. Define the linear map
\begin{align*}
   V\colon \mathbb{C}^m \to \Hil, b \mapsto \sum_i v_i b_i
\end{align*}
and consider the \emph{generalized eigenvalue problem}:
\begin{align}\label{eq:sep}
  V^\dagger H V b_i = \tilde \lambda_i V^\dagger V b_i
\end{align}
Here $b_i \in \mathbb{C}^m$ and $\tilde \lambda_i \in \mathbb{R}$. 
As Eq.\,\cref{eq:sep} intends to approximate the eigenvalue of $H$ in $\mathcal{E}$, we decompose the guess vector map $V$ into signal and noise
\begin{equation}
\begin{alignedat}{2}
   S &:= P_{\mathcal{E}} V &\qquad\qquad N &:= P_{\mathcal{E}^\perp} V.
\end{alignedat}
\end{equation}
Here, $P_{\mathcal{E}}$ is the orthogonal projection operator onto $\mathcal{E}$ and $\mathcal{E}^{\perp}$ is the orthogonal complement.
To describe the accuracy of the subspace approximation, we introduce an \emph{error measure}
\begin{align}
  \varepsilon^{(V)}_{(H, \mathcal{E})}: \mathcal{B}(\mathbb{R}) \to \mathbb{R},  I \mapsto \Tr [N^\dagger P^{(H)}(I)N],
\end{align}
where $P^{(H)}$ is the projection-valued measure (PVM) such that $ H = \int \lambda d P^{(H)}(\lambda)$. The error measure is finite as the trial vectors $v_i$ are in 
the form domain of $H$. 

As a numerical protocol is typically subject to multiple error sources, we further relax the problem description and assume a generalized eigenvalue problem
\begin{align}\label{eq:geneigprob}
    A b_i = \tilde \lambda_i B b_i 
\end{align}
with 
\begin{align}
   A = V^\dagger H V + \delta A \qquad B = V^\dagger V +  \delta B.
\end{align}
Here $\delta A, \delta B \in \mathbb{C}^{m\times m}$ are self-adjoint matrices describing additional error sources. They may arise when
computing the matrix elements $(v_i, H v_j)$ or by processing noisy data. 
Overall we assume a signal noise decomposition
\begin{align*}
   A = S^\dagger H S + \mathcal{N}^{(A)} \qquad B = S^\dagger S + \mathcal{N}^{(B)}, 
\end{align*}
where $\mathcal{N}^{(A)}= N^\dagger H N + \delta A $ and $\mathcal{N}^{(B)} = N^\dagger N + \delta B$.
To avoid degenerate cases, we further assume that $B$ is positive definite. 
To compress notation we denote generalized eigenvalue problems as in Eq.\,\cref{eq:geneigprob} by $[A,B]$ and call $B$ the \emph{weight matrix}.

\begin{remark}
   We highlight two advantages of our PVM-based error measure over the standard approach of using the residual norm $\|(H-\tilde{\lambda}_i) b_i\|$.
   First, for unbounded operators (e.g. differential operators), trial vectors often lie in the form domain $\mathcal{Q}(H)$ but not in the operator domain $\mathcal{D}(H)$ (e.g. finite elements with limited regularity). In such cases, the residual norm is undefined, whereas our measure remains well-defined.
   Second, the residual is impractical for extensive derivations, as it conflates the spectral location of spurious contributions with their magnitude. 
   In the derivations that follow, the measure $\varepsilon$ allows us to separate the subspace error from a spectral-distance weighted error on the eigenvalues. 
   This yields a structural clarity that is advantageous for subsequent applications and error budgeting.
\end{remark}

\section{Spectral inequalities}
\label{sec:inequalities}
We now present inequalities that bound the error of the approximate eigenvalues through the spectral measure. 
To simplify notation, we denote by $\tilde \lambda_1 \geq \cdots \geq \tilde \lambda_m $ 
the eigenvalues of the generalized eigenvalue problem $[A,B]$, and by $\lambda_1\geq \cdots \geq \lambda_m $ the eigenvalues of $H$ in $\mathcal{E}$. 
Meanwhile, $\lambda_i(M)$ shall denote the $i$-th largest eigenvalue of the self-adjoint matrix $M$.
We also write for the error measure $\varepsilon^{(V)}_{(H, \mathcal{E})}$ simply $\varepsilon$, as the context is clear.
To account for indefinite behavior in error matrices $\delta A - \tilde \lambda_i \delta B$ we use the functions
\begin{align*}
   \lambda_{\min}^*(M):=\min \{\lambda_m(M), 0 \} \quad \text{ and }  \quad \lambda_{\max}^*(M):=\max \{\lambda_1(M), 0 \}
\end{align*}
and give compact spectral inequalities. 

\begin{theorem}
    \label{thm:master}
Let $[A,B]$ be a generalized eigenvalue problem as in \cref{sec:descr}. Assume $[A,B]$ is well conditioned in the sense 
   \begin{align}\label{eq:wellCond}
      \lambda_m(B) > \lambda_1(\mathcal{N}^{(B)}), 
   \end{align}
   where $m = \dim(\mathcal{E})$.
Then the eigenvalues of $[A,B]$ and $H$ in $\mathcal{E}$ have:
    \begin{align}
 \frac{\int_{\lambda < \tilde \lambda_i } (\lambda - \tilde \lambda_i) d \varepsilon(\lambda) + \lambda_{\min}^* (\delta A - \tilde \lambda_i \delta B ) }{\lambda_m(B) - \lambda_1(\mathcal{N}^{(B)})} \leq \tilde \lambda_i -\lambda_i \leq \frac{\int_{\lambda > \tilde \lambda_i } (\lambda - \tilde \lambda_i) d \varepsilon(\lambda) + \lambda_{\max}^* ( \delta A- \tilde \lambda_i  \delta B ) }{\lambda_m(B) - \lambda_1(\mathcal{N}^{(B)})}. \label{eq:main}
    \end{align}
\end{theorem}

This theorem is the central result of our framework and addresses the limitations of classical theory. 
The bound is not directly dependent on an eigenvalue gap, but instead characterized by an 
``off-band" error measure $\varepsilon(\lambda)$.  This PVM-based error measure quantifies the subspace error even with respect to unbounded operators,
which classical matrix norms cannot. The inequalities cleanly separate
subspace error from other error sources (described by $\delta A, \delta B$) arising from additional noise or numerical subroutines. 
Thus, \cref{thm:master} serves as a master theorem to describe the accuracy of \emph{subspace protocols}; a class of numerical methods that we define in \cref{sec:SubspaceProtocol}.
We give the proof of the theorem in \cref{sec:deri_specineq}.

In many applications, the spectral region of interest is specified by an interval. For this case we denote with $\mathcal{E}[a,b]$ the spectral subspace of $H$ that is spanned by all eigenvectors 
with eigenvalues in the interval $[a,b]$. Then we obtain a more practical version of \cref{thm:master}.

\begin{corollary}\label{cor:band}
   Let $[A,B]$ be a generalized eigenvalue problem as in \cref{sec:descr} where the spectral subspace is a band $\mathcal{E} = \mathcal{E}[a,b]$.
   Assume: i) that the generalized eigenvalue problem is well-conditioned as in Eq.\,\cref{eq:wellCond}, ii) that $H$ has no essential spectrum in the interval $[a,b]$, and iii) that $\tilde \lambda_i \in [a,b]$ for all $i$.
   Then the eigenvalues of $[A,B]$ satisfy:
   \begin{align}
      \frac{\int_{\lambda < a } (\lambda - \tilde \lambda_i) d \varepsilon(\lambda) + \lambda_{\min}^* (\delta A - \tilde \lambda_i \delta B ) }{\lambda_m(B) - \lambda_1(\mathcal{N}^{(B)})} \leq \tilde \lambda_i -\lambda_i \leq \frac{\int_{\lambda > b } (\lambda - \tilde \lambda_i) d \varepsilon(\lambda) + \lambda_{\max}^* (\delta A - \tilde \lambda_i \delta B ) }{\lambda_m(B) - \lambda_1(\mathcal{N}^{(B)})}. 
   \end{align}
\end{corollary}

Assumption \emph{ii)} ensures that the error measure $\varepsilon$ has no support in $[a,b]$, while assumption \emph{iii)} allows the integration bounds of \cref{eq:main} to be lowered to $a$ and increased to $b$ respectively.
\Cref{sec:related_specineq} provides further spectral inequalities related to \cref{thm:master} and \cref{cor:band}.

These spectral inequalities capture the critical property of \emph{energy concentration} to characterize the quality of a trial space. 
While a numerical protocol may not perfectly confine the trial vectors to the 
spectral subspace $\mathcal{E}[a,b]$, we expect an effective routine to generate trial vectors that are highly localized around it. Formally, this implies that the spectral measure of the trial space
\begin{align}
    \alpha(I) := \Tr [V^\dagger P^{(H)}(I) V]
\end{align}
is strongly weighted within the interval $[a,b]$ with fast decreasing tails outside. 
A central goal of a subspace protocol (as defined in \cref{sec:SubspaceProtocol}) is to generate trial spaces $V$ that exhibit this behavior by construction. Since the integrals in \cref{cor:band} weight 
the error measure by the distance to the interval, this concentration ensures small error bounds.

\section{Dimension detection and algorithmic stability}
\label{sec:SubspaceProtocol_intro}

In the previous two sections, we assumed that the dimension of the trial space
coincides with the dimension of the spectral subspace. 
But in practice, the dimension of the spectral subspace is typically not known.

Indeed, across a wide range of fields computational methods exhibit certain artifacts
related to a mismatch in dimensionality,
without being fully explained or predicted by the theory. 
Literature on spectral approximation and finite element methods 
reports on the difficult-to-predict phenomenon of \emph{spectral pollution}, leading to the failure of the Galerkin 
approximation \cite{Boffi_2010, boulton_spectral_2016}.
Statistical learning theory faces the \emph{singularity problem} in a linear discriminant analysis, if the number of data samples 
is smaller than the dimension of the feature space \cite{Sharma2015}. 
Finally, numerical spectral analysis frequently suffers from ill-conditioning and the occurrence of \emph{spurious eigenvalues} in subspace methods \cite{Parlett1998, CullumWilloughby1985}.

Within our framework, \emph{all} these related phenomena are explained by a poor trial space whose dimensionality 
exceeds the dimension of the subspace that is to be approximated.\footnote{ In the context of linear discriminant analysis, our argument considers the feature space of the data (that is yet to be reduced)
as analogous to the trial space.}

Established strategies to address these artifacts include the introduction of additional noise 
to regularize the matrix problem and posterior filtering of spurious eigenvalues \cite[Chapter 4]{CullumWilloughby1985} \cite{Morningstar2024}. 
However, an ill-conditioned eigenvalue of a generalized eigenvalue problem can still cause drastic noise amplification on well-conditioned 
eigenvalues \cite{stewart_perturbation_1978}. In particular, spurious contributions 
can indeed pollute the whole spectrum. Therefore, we suggest abandoning the practice of posterior filtering. 
Instead, we avoid spurious eigenvalues entirely and suggest 
a spectral technique to detect the hidden dimensionality followed by a refinement 
of the trial space.

\subsection{Formalization of subspace protocols}
\label{sec:SubspaceProtocol}
We formalize this solution by introducing \emph{subspace protocols}, a general class of numerical methods compatible with our framework, and specify their algorithmic operation.

\begin{definition}[subspace protocols]
   \label{def:SubspaceProtocol}
    A numerical routine $\mathsf{P}$ is called a \emph{subspace protocol} if: given a self-adjoint operator
    $H$ on a Hilbert space $\Hil$ and a spectral subspace $\mathcal{E}$ of $H$, the protocol $\mathsf{P}$ generates for any given guess dimension $M\in \mathbb{N}$ a generalized eigenvalue problem:
      \begin{align}
   A_M b_k = \tilde \lambda_k B_M b_k \label{eq:GEP_simple},
      \end{align}
      where $A_M, B_M \in \mathbb{C}^{M\times M}$ are self-adjoint matrices, $b_k \in \mathbb{C}^M$ and $\tilde \lambda_k \in \mathbb{R}$.
    For a fixed pair $(H, \mathcal{E})$ we write $\mathsf{P}_{(H, \mathcal{E})}(M)$ for the generalized eigenvalue problem generated by $\mathsf{P}$ at a guess dimension $M$.
\end{definition}

As \cref{def:SubspaceProtocol} shows, this work remains agnostic to the detailed mechanism 
with which a subspace protocol computes the finite matrix problems. 
Analogous to \cref{sec:descr}, we assume a signal-noise decomposition the for generalized eingenvalue problems of a subspace protocol:
\begin{align}
   A_M & = V^\dagger_M H V_M + \delta A_M\\
   B_M & = V^\dagger_M V_M + \delta B_M,
\end{align}
where $V_M:\mathbb{C}^M \rightarrow \Hil $ is a linear map, and $\delta A_M, \delta B_M \in \mathbb{C}^{M\times M}$ are self-adjoint matrices capturing errors
in addition to the subspace projection. 
We define $S_M := P_{\mathcal{E}} V_M$ and $N_M := P_{\mathcal{E}^\perp} V_M$, where $P_{\mathcal{E}}$ is the orthogonal projection operator onto $\mathcal{E}$.
The matrix
\begin{align}
      \mathcal{N}^{(B)}_M  := N_M^\dagger N_M + \delta B_M
\end{align}
quantifies the deviation from the ideal projection onto $\mathcal{E}$ and we call $\mathcal{N}^{(B)}_M$ the \emph{noise weight} at trial dimension $M$.\\

During the operation of a subspace protocol $\mathsf{P}$ the operator $H$ and the target subspace $\mathcal{E}$ remain fixed, while the trial dimension $M$ and 
the assumed dimension $m$ of the spectral subspace $\mathcal{E}$ are variable input parameters. If unknown, we propose 
to detect the dimension of $\mathcal{E}$ through a significant drop in the spectrum of the weight matrix $B_M$
or by applying a noise threshold $\epsilon_{\text{th}} > 0$. \cref{thm:dim_lower_bound} justifies this method. 
In particular, we propose \cref{alg:eps_dim_red} to obtain approximations of the eigenvalues of $H$ in $\mathcal{E}$.

\begin{algorithm}[H]
\caption{Operation of a subspace protocol $\mathsf{P}_{(H, \mathcal{E})}(M, \epsilon_{\text{th}})$}
\label{alg:eps_dim_red}
\begin{enumerate}
    \item Obtain from $\mathsf{P}_{(H, \mathcal{E})}(M)$ the matrices $A_M, B_M \in \mathbb{C}^{M \times M}$, where $M$ is the current trial dimension. Diagonalize $B_M$ and obtain its eigenvalues.
    \item Define the \emph{detected dimensionality} $m$ as the number of eigenvalues of $B_M$ exceeding the threshold $\epsilon_{\text{th}}$. If $m = M$, increase the trial dimension $M$ and repeat the process.
    \item Let $U_m \in \mathbb{C}^{M \times m}$ be the matrix whose columns are the $m$ leading eigenvectors of $B_M$. Compute projected matrices $A_M^m := U_m^\dagger A_M U_m$ and $B_M^m := U_m^\dagger B_M U_m$. The generalized eigenvalue problem $[A_M^m, B_M^m]$ is of dimension $m$ and satisfies $\lambda_m(B_M^m) > \epsilon_{\text{th}}$.
    \item Compute the generalized eigenvalues of $[A_M^m, B_M^m]$. They approximate the eigenvalues of $H$ in the spectral subspace $\mathcal{E}$. 
\end{enumerate}
\end{algorithm}
If in step~2, all eigenvalues of $B_M$ are smaller than $\epsilon_{\text{th}}$, a zero subspace is detected, and we set $m = 0$. This can occur either when the target spectral subspace $\mathcal{E}$ is trivial ($\mathcal{E} = {0}$), or when the weight matrix $B_M$ is too poor to distinguish signal from noise.
The refinement in step $3$ restricts the generalized eigenvalue problem to the well-conditioned subspace of $B_M$. This refinement improves numerical stability and allows to use a trial dimension $M$ exceeding $\dim(\mathcal{E})$.
The accuracy of the approximate eigenvalues is estimated by \cref{thm:master} assuming that $m$ is indeed the dimension of $\mathcal{E}$.
The spectral technique to detect the dimension is motivated from a stronger class of subspace protocols,
that also provides bounds on the noise weights: 

\begin{definition}[$\epsilon$-subspace protocols]
\label{def:eps_SubspaceProtocol}
      Let $\mathsf{P}$ be a subspace protocol as in \cref{def:SubspaceProtocol}, and let 
       $\{\epsilon_M\}_{M=1}^\infty$ be a sequence of non-negative real numbers such that for all pairs 
       $(H,\mathcal{E})$ the noise weights  of $\mathsf{P}_{(H,\mathcal{E})}$ have 
      \begin{align}
         \| \mathcal{N}^{(B)}_M \| \leq \epsilon_M \qquad \text{ for all } M \in \mathbb{N}.  \label{eq:eps_M}
     \end{align}
     Then, we call the pair $(\mathsf{P},\{\epsilon_M\}_{M=1}^\infty )$ an \emph{$\epsilon$-subspace protocol} and denote it by $\mathsf{P}^{\{\epsilon_M\}}$.
\end{definition}

For an $\epsilon$-subspace protocol we define a modification of \cref{alg:eps_dim_red} replacing the detection threshold $\epsilon_{\text{th}}$
by the noise bounds $\epsilon_M$:
\begin{align*}
   \mathsf{P}^{\{\epsilon_M\}}_{(H,\mathcal{E})}(M) := \mathsf{P}_{(H,\mathcal{E})}(M, \epsilon_M)
\end{align*}

For $\epsilon$-protocols we obtain an analytic guarantee on the detected dimension of the spectral subspace. 

\begin{theorem}
    \label{thm:dim_lower_bound}
 The dimension $m$ detected by an  $\epsilon$-subspace protocol  $\mathsf{P}^{\{\epsilon_M\}}_{(H,\mathcal{E})}(M)$ is a lower bound for the true dimension of the spectral subspace $\mathcal{E}$,
    \begin{align}
 \dim \mathcal{E} \geq m. \label{eq:dim_lower_bound}
    \end{align}
\end{theorem}

The simple proof of this useful statement relies on Weyl's inequality \cref{eq:Weyl}: 
 Let $B_M$ be the weight matrix obtained in step~1 of \cref{alg:eps_dim_red} and let $m$ be the largest index such that $\lambda_m(B_M) > \epsilon_M$. For $m=0$ no signal is detected and \cref{eq:dim_lower_bound} is 
 trivially satisfied. For $m\geq 1$, Weyl's inequality \cref{eq:Weyl} for the eigenvalues of self-adjoint matrices gives
    \begin{align*}
        \lambda_m(S_M^\dagger S_M) & = \lambda_m(B_M - \mathcal{N}_M^{(B)}) \geq \lambda_m(B_M) - \lambda_1(\mathcal{N}_M^{(B)} )> 0.
    \end{align*}
 The last inequality follows from $\lambda_m(B_M) > \epsilon_M$ at the detected dimension $m$. 
 Gram matrices have $\operatorname{rank}(S_M^\dagger S_M)= \dim (\operatorname{Ran} S_M)$ and therefore the range of $S_M$ is at least $m$-dimensional. By construction $\operatorname{Ran} S_M \subset \mathcal{E}$. Thus, $\dim \mathcal{E} \geq m$
and \cref{thm:dim_lower_bound} is proven.
\begin{remark}
In the context of \emph{multiple source detection}, the problem of determining $m$ is known as \emph{model order selection}
and is similarly solved by analyzing the eigenvalues of a sample covariance matrix \cite{zhao_detection_1986,WaxKailath1985}.
The foundational approach in that field is to apply the Akaike information criterion (AIC) or the minimum description length (MDL) to find a
most likely $m$. Through our framework this statistical problem finds 
a deterministic analog in the context of operator approximation. 
\cref{thm:dim_lower_bound} provides a criterion towards rigorous dimension detection based on norm bounds 
on noise weights. This allows spectral dimension detection techniques to be applied to a wider range of numerical applications
where statistical assumptions do not necessarily apply. 
\end{remark} 

Moving beyond the lower bound of \cref{thm:dim_lower_bound}, we derive an algebraic condition to guarantee that the detected dimension coincides with 
$\dim \mathcal{E}$. 
For \cref{alg:eps_dim_red} to correctly identify $m^*:= \dim \mathcal{E} $, the $m^*$-th eigenvalue of the weight matrix must exceed the noise threshold. 
Using Weyl's inequality we obtain a sufficient condition on the \emph{signal strength} $\lambda_{m^*}(S_M^\dagger S_M)$:
\begin{align*} 
    \lambda_{m^*}(B_M) & \geq \lambda_{m^*}(S_M^\dagger S_M) + \lambda_{m^*}(\mathcal{N}_M^{(B)}) \stackrel{!}{>} \epsilon_M. 
\end{align*} 
In the case of pure subspace noise ($\mathcal{N}_M^{B} = N_M^\dagger N_M \geq 0$), this leads to requiring $\lambda_{m^*}(S_M^\dagger S_M) > \epsilon_M$. In the general case, 
where the noise is indefinite but bounded by $\| \mathcal{N}_M^{B} \| \leq \epsilon_M $, this threshold is lifted to $2 \epsilon_M$.
Overall, we identify
\begin{align} \label{eq:detect_cond}
   \lambda_{m^*}( S_M^\dagger S_M) \stackrel{!}{\gtrsim} \epsilon_M 
\end{align}
as a sufficient algebraic condition for correct dimension detection. 
Conversely, \cref{eq:detect_cond} is also a necessary condition for the trial space to exhibit high energy concentration and to ensure small error bounds in \cref{sec:inequalities}.

\begin{remark}
   When discussing sufficient conditions for accurate dimension detection, we also point to the spike model of random matrix theory \cite[Section 3.3]{BandeiraSingerStrohmer2025} \cite{BaikSilverstein2006}.
   Here $B_M$ is a signal covariance matrix $S_M^\dagger S_M$ that is perturbed by 
   a random matrix $\mathcal{N}_M^{B}$ describing noise. The eigenvalues of $\mathcal{N}_M^{B}$ are random variables characterized by a distribution (such as Marchenko-Pastur) and
   the edge of this distribution function corresponds to $\epsilon_M$. The BBP phase transition characterizes the necessary signal strength $\lambda_{m^*}(S_M^\dagger S_M)$ such that $B_M$ exhibits indeed eigenvalues 
   spiking out of the threshold $\epsilon_M$  \cite{BaikBenArousPeche2005}. However, these results are asymptotic $(M \to \infty)$ and assume a uniform noise distribution, which is often impractical towards applications. 
\end{remark}

\begin{remark}
   Eq.\,\cref{eq:detect_cond} serves a practical starting point to derive stability conditions for subspace protocols. 
   In the special case of filter diagonalization, Eq.\,\cref{eq:detect_cond} was used to derive a minimal amplitude guaranteeing that the correct number of frequencies is detected \cite[Section\,2.3]{stroschein2025groundexcitedstateenergiesanalytic} \cite[Section\,5.3.3]{stroschein2024prolatespheroidalwavefunctions}.
   Through an alternant matrix this condition is also related to the density of frequencies. 
   This provides a deterministic parallel to the performance analysis of signal processing routines such as MUSIC and ESPRIT. 
   Guarantees in that field are derived from a probabilistic noise model and give asymptotic results on the eigenvalue estimates. 
   However, the statistical analyses \cite{StoicaNehorai1989, RaoHari1993} similarly conclude that the stability is dependent on the signal to noise ratio (analogous to our minimal frequency-amplitude) 
   and source separation (our density of frequencies). 
   Modern stability analyses also use the properties of a Vandermonde matrix (in place of our alternant matrix) to characterize the source separation condition non-asymptotically \cite{LiaoFannjiang2016}. 
   Future work could focus on formally embedding MUSIC and ESPRIT within our generalized framework to enable a more direct comparison of the methods. This would also bridge the gap between deterministic and asymptotic error quantification. 
\end{remark}

To improve the reliability of the dimension detection and the accuracy of the eigenvalues, we recommend to choose a larger trial dimension $M$ exceeding
$\dim(\mathcal{E})$. A larger trial space is more expressive and can thus better represent the spectral subspace. To make this point more formal consider the algorithm $\mathsf{P}_{(H, \mathcal{E})}(M, m)$ 
that takes the assumed dimension $m$ of $\mathcal{E}$ directly as input; thus $\mathsf{P}_{(H, \mathcal{E})}(M, m)$  is identical to \cref{alg:eps_dim_red}
up to omitting step\,2. Then we have: \emph{the conditioning $\lambda_m(B_M^m)^{-1}$ of a generalized eigenvalue problem generated by $\mathsf{P}_{(H, \mathcal{E})}(M, m)$ is monotone decreasing in the guess dimension $M$.}\\
This statement follows from the fact that for all trial dimensions $M,M'$ with $M' > M$ the weight matrix $B_M$ is the $M$-th leading principal submatrix of the weight matrix $B_{M'}$.
Then the variational principle \cref{eq:maxmin} has 
\begin{align*}
   \lambda_m(B_M) & = \max_{S_m \subset \mathbb{C}^{M}} \min_{x \in S_m} \frac{(x,B_M x)}{(x,x)} = \max\limits_{\substack{S_m \subset \mathbb{C}^{M'} \\ S_m \subset [e_1,\dots,e_M]}} \min_{x \in S_m} \frac{(x,B_{M'} x)}{(x,x)}\\
   & \leq  \max\limits_{\substack{S_m \subset \mathbb{C}^{M'} }} \min_{x \in S_m} \frac{(x,B_{M'} x)}{(x,x)} = \lambda_m(B_{M'}).
\end{align*}
In particular, the conditioning $\lambda_m(B_M)^{-1}$ of $\mathsf{P}_{(H, \mathcal{E})}(M, m)$ is monotone decreasing in the guess dimension $M$.

However, the advantage of increasing $M$ is not unconditional. Contributions from noise (captured by $\epsilon_M$) are typically also increasing with $M$. 
To address this conflict of interest, we propose informing the choice of $M$ through a noise-to-signal ratio defined as  $\epsilon_M / \lambda_m(B_M^m)$.
The number of guess vectors should be increased only as long as this ratio continues to decrease significantly. A decreasing noise-to-signal ratio indicates that the improved signal representation (captured by $\lambda_{m}(S_M^\dagger S_M)$) outweighs the adverse effect of increased noise.

Finally, once the correct dimension $m$ is detected and the generalized eigenvalue problem $[A_M^m, B_M^m]$ satisfies $\lambda_m(B_M^m) > \| \mathcal{N}^{(B)}_M \|$, \cref{thm:master} estimates the accuracy of the computed eigenvalues. 
Strictly speaking \cref{thm:master} applies with respect to the error measure $\varepsilon^{(V_M^m)}_{(H, \mathcal{E})}$ of the refined trial space $V_M^m = V_M U_m$. However, the refined error measure is estimated by the unrefined:
\begin{align*}
   \underbrace{\Tr [ (P_{\mathcal{E}^\perp} {V_M^m})^\dagger  P^{(H)}(I) (P_{\mathcal{E}^\perp} V_M^m) ]}_{\varepsilon^{(V_M^m)}_{(H, \mathcal{E})}} \leq \underbrace{ \Tr [ (P_{\mathcal{E}^\perp} {V_M})^\dagger  P^{(H)}(I) (P_{\mathcal{E}^\perp} V_M) ]}_{\varepsilon^{(V_M)}_{(H, \mathcal{E})}}
\end{align*}
for all $I \in \mathcal{B}(\mathbb{R})$. This follows directly from the definition of $V_M^m$ as a projection of $V_M$.
Thus, the inequalities of \cref{sec:inequalities} also hold with respect to the unrefined error measure.

\section{Conclusion and outlook}
\label{sec:outlook}

In this work, we introduced a non-asymptotic approximation framework for subspace methods that diverges from classical gap-dependent (Davis-Kahn) and operator-convergence (Kato/Chatelin) theories. Our contributions include: (1) a PVM-based 
error measure and a master theorem quantifying multiple error sources for unbounded operators; and (2) a formal framework for subspace protocols that provides a rigorous foundation for \emph{spectral dimension detection}, resolving 
frequent artifacts such as spectral pollution and enabling non-asymptotic stability analysis. 

This framework reveals that accurate local approximation hinges on effective distillation of spectral degrees of freedom from noise.
By amplifying the targeted subspace and suppressing spurious contributions, a subspace protocol ensures reliable dimension detection and yields ``good'' trial spaces.
The quality of the trial space is characterized by its concentration in the spectral region of interest.
Under the PVM-based error measure, such concentration directly translates to
sharp error bounds in \cref{sec:inequalities}. 

A long-term goal is to develop a unified approximation framework describing all error contributions, spanning from infinite-dimensional formalism to finite-precision arithmetic. Such a theory would inform the optimal allocation of computational resources
and facilitate efficient algorithms operating with limited degrees of freedom. The work presented here contributes to this goal by offering a common language to unify results from various approximation routines in one framework. In particular, the new error measure and integrated spectral
inequalities appear to be essential in bridging the gap between the rigorous formalism of unbounded operators and the practical realities of numerical computation.

This approach has already yielded practical results in the context of signal processing and quantum computation. In the analysis of sampled prolate filter diagonalization, it provided an unprecedented understanding of how computational accuracy and spectral density relate to
necessary observation time. This result can now inform the minimal simulation time required for larger approximation routines, 
particularly those involving a quantum device~\cite{stroschein2025groundexcitedstateenergiesanalytic, Reiher2017}.

Given its generality, this formalism allows for application to a wide range of local approximation methods. Other candidates for future analysis as subspace protocols include:

\begin{itemize}
   \item Krylov subspace methods (e.g. Lanczos/Arnoldi): Our framework offers a new perspective on two challenges. First, for 
   implementations without full reorthogonalization, spectral dimension detection (Alg. 1) can be applied to the weight matrix 
   $B_M = V_M^\dagger V_M$ to \emph{a priori} prevent spurious eigenvalues, contrasting with classic \emph{a posteriori} filters \cite{CullumWilloughby1985}. 
   Second, \cref{thm:master} provides a gap-independent analytical tool beyond standard perturbation theory~\cite{Saad2011}.
   \item Jacobi-Davidson methods: These iteratively expand the trial space by solving a preconditioned ``correction equation" \cite{SleijpenVanDerVorst2000}.
   Our analysis of the noise-to-signal ratio ($\epsilon_M / \lambda_m(B_M^m)$) provides a rigorous diagnostic to dynamically control the quality of this preconditioning step, balancing the ``noise" ($\epsilon_M$) of an inexact solve against the ``signal" gain.
   \item Contour integral methods such as FEAST \cite{Polizzi2009}: Here, the quadrature errors from approximating the contour integral
   map directly to additional error matrices $\delta A$ and $\delta B$. Spectral dimension detection can serve to identify and avoid mismatches in trial dimension, a challenge particularly evident in hybrid methods such as DMRG[FEAST]~\cite{Baiardi2022}.
   \item Tensor Network Methods (e.g. DMRG): The discarded weight in a reduced density matrix is a highly successful but heuristic proxy for eigenvalue error~\cite{Schollwoeck2005, Baiardi2020}. Our framework offers a path to make this step rigorous, providing a formal criterion for choosing the bond dimension $m$ and a non-asymptotic bound on the resulting error.
   \item Signal processing routines such as MUSIC and ESPRIT: The analysis of these methods traditionally relies on asymptotic guarantees and probabilistic noise models~\cite{StoicaNehorai1989, RaoHari1993}.
   Our approach enables their embedding into a deterministic, non-asymptotic context. This allows classic stability conditions—such as ``SNR" and ``source separation"—to be mapped directly to the algebraic stability conditions derived in this work.
\end{itemize}

\section*{Acknowledgements}
The author thanks Markus Reiher for his continued support and trust. Special thanks are due to Gian Michele Graf for numerous 
inspiring conversations. Financial support from the Swiss National Science Foundation (Grant No. 200021\_219616) is gratefully acknowledged.

\begin{appendix}

\section{Derivations}

\subsection{Preliminaries}

The proof of the spectral inequalities in \cref{thm:master} relies on the max-min principle for generalized eigenvalues.  
\begin{lemma}[Min-max principle]
   Let $A,B \in \mathbb{C}^{m\times m}$ be self-adjoint and $B$ be positive definite. Let $\lambda_1[A,B]\geq \lambda_2[A,B] \geq \dots \geq \lambda_m[A,B] $  be the eigenvalues of the generalized eigenvalue problem $[A,B]$.
   The $n$-th eigenvalue satisfies:
   \begin{align}
    \lambda_n[A,B]  & = \max_{\mathcal{S}_n} \min_{x \in \mathcal{S}_n} \frac{( x, A x  )}{(x, B x )} \label{eq:maxmin} \\
                 & = \min_{\mathcal{S}_{n-1}} \max_{x \in \mathcal{S}_{n-1}^\perp } \frac{( x, A x  )}{(x, B x )},  \label{eq:minmax}
\end{align}
where $\mathcal{S}_n$ is a $n$-dimensional subspace of $\mathbb{C}^m$. 
\end{lemma}

In our derivations we also make frequent use of Weyl's inequalities. 
\begin{lemma}[Weyl's Inequality]
   Let $A, B \in \mathbb{C}^{m\times m}$ be self-adjoint matrices. For indices $i, j \geq 1$, the eigenvalues of the sum $A+B$ satisfy:
   \begin{align}
      \lambda_{i+j-m}(A + B) \geq \lambda_i(A) + \lambda_j(B) \geq \lambda_{i+j-1}(A+B), \label{eq:Weyl}
   \end{align}
   where $\lambda_j(\cdot)$ denotes the $j$-th largest eigenvalue and terms with indices outside $\{1, \dots, m\}$ are omitted.
\end{lemma}

\subsection{Derivations of spectral inequalities}
\label{sec:deri_specineq}

\cref{thm:master} is derived from \cref{prop:var_principle}.
\begin{proposition}\label{prop:var_principle}
Let $[A,B]$ be a generalized eigenvalue problem as in \cref{sec:descr}. Assume $[A,B]$ is well conditioned in the sense 
   \begin{align}\label{eq:wellCond2}
      \lambda_m(B) > \lambda_1(\mathcal{N}^{(B)}). 
   \end{align}
 Let $\tilde \lambda_i$ denote the eigenvalues of $[A,B]$ and $\lambda_i$ the eigenvalues of $H$ in $\mathcal{E}$
 in descending order. The difference of the eigenvalues satisfies:
    \begin{align}
 \tilde \lambda_i -\lambda_i \leq \frac{\lambda_1 \left ( \int_{\lambda > \tilde \lambda_i } (\lambda - \tilde \lambda_i) d N^\dagger P^{(H)}(\lambda) N \right ) + \lambda_{\max}^*  ( \delta A - \tilde \lambda_i  \delta B  )}{\lambda_m(B) - \lambda_1(\mathcal{N}^{(B)})}   \label{eq:prop_upper} \\
 \frac{ \lambda_m \left ( \int_{\lambda < \tilde \lambda_i } (\lambda - \tilde \lambda_i) d N^\dagger P^{(H)}(\lambda) N  \right ) +  \lambda_{\min}^*   ( \delta A - \tilde \lambda_i  \delta B  ) }{\lambda_m(B) - \lambda_1(\mathcal{N}^{(B)})} \leq \tilde \lambda_i -\lambda_i    \label{eq:prop_lower}
    \end{align}
\end{proposition}

\begin{proof}[Proof of \cref{prop:var_principle}]
   By Weyl's inequality \cref{eq:Weyl},
   \begin{align}
      \lambda_m(S^\dagger S) \geq \lambda_m(B) - \lambda_1(\mathcal{N}^{(B)}) > 0.
   \end{align}
   The last inequality follows from assumption\,\cref{eq:wellCond2}. Gram matrices have $\operatorname{rank}(S^\dagger S) = \dim( \operatorname{range}S)$.
   From $ \operatorname{range}S \subseteq \mathcal{E}$ and $\dim( \operatorname{range}S) = \dim(\mathcal{E})$, it follows that the range of $S$ is the spectral subspace $\mathcal{E}$.
   In particular, the eigenvalues of the generalized eigenvalue problem $[S^\dagger H S, S^\dagger S]$ are the eigenvalues of $H$ in $\mathcal{E}$. \\
   
   The rest of the proof relies on a comparison of the eigenvalues of $[A,B]$ with $[S^\dagger H S, S^\dagger S]$ using the min-max theorem. We start with the lower inequality\,\cref{eq:prop_lower}.
   Let $\mathcal{S}_{i-1} \subset \mathbb{C}^m $ be the subspace spanned by the leading $i-1$ generalized eigenvectors of $[A,B]$. 
   For $x \in \mathcal{S}_{i-1}^{\perp}$ arbitrary, we estimate the Rayleigh quotient of $[S^\dagger H S, S^\dagger S]$: 
      \begin{align*}
   \frac{(x,S^\dagger HSx)}{(x,S^\dagger S x)} &= \frac{(x,(A - N^\dagger H N - \delta A) x)}{(x,(B-\mathcal{N}^{(B)})x)} \leq \frac{(x,(\tilde \lambda_i B - N^\dagger H N- \delta A) x)}{(x,(B-\mathcal{N}^{(B)}) x)} \\
        & = \tilde \lambda_i + \frac{(x,(\tilde \lambda_i N^\dagger N + \tilde \lambda_i \delta B  - N^\dagger H N - \delta A) x)}{(x,(B-\mathcal{N}^{(B)})x)} \\
        & = \tilde \lambda_i + \frac{(x, \left (\int (\tilde \lambda_i -\lambda) d N^\dagger P^{(H)}(\lambda) N + \tilde \lambda_i \delta B - \delta A \right ) x)}{(x,(B-\mathcal{N}^{(B)})x)} \\
        & \leq\tilde \lambda_i - \frac{(x, \left (\int_{\lambda < \tilde \lambda_i } (\lambda - \tilde \lambda_i) d N^\dagger P^{(H)}(\lambda) N + \delta A - \tilde \lambda_i \delta B \right ) x)}{(x,(B-\mathcal{N}^{(B)})x)} \\
        & \leq \tilde \lambda_i - \frac{ \lambda_{m}  \left (\int_{\lambda < \tilde \lambda_i } (\lambda - \tilde \lambda_i)d N^\dagger P^{(H)}(\lambda) N + \delta A- \tilde \lambda_i \delta B \right ) }{(x,(B-\mathcal{N}^{(B)})x) }
    \end{align*}
    The first inequality applies $(x,Ax) \leq \tilde \lambda_i (x,Bx)$. In the second, we split the integral into two regions and drop the part $\lambda > \tilde \lambda_i$. The third inequality uses that $B - \mathcal{N}^{(B)}$ is positive definite.
    We continue our estimate with Weyl's inequality:
      \begin{align*}
         \cdots & \leq \tilde \lambda_i - \frac{ \lambda_{m}  \left (\int_{\lambda < \tilde \lambda_i } (\lambda - \tilde \lambda_i)d N^\dagger P^{(H)}(\lambda) N \right )+ \lambda_{m}  \left ( \delta A- \tilde \lambda_i \delta B \right ) }{(x,(B-\mathcal{N}^{(B)})x) }\\
         & \leq \tilde \lambda_i - \frac{ \lambda_m \left (\int_{\lambda < \tilde \lambda_i } (\lambda - \tilde \lambda_i)d N^\dagger P^{(H)}(\lambda) N \right ) +  \lambda_{\min}^* \left ( \delta A- \tilde \lambda_i \delta B \right ) }{\lambda_m(B) - \lambda_1(\mathcal{N}^{(B)}) }
    \end{align*}
    The last inequality uses that the numerator is non-positive (due to $\lambda_{\min}^*(\cdot) \leq 0$ and the negative integral).
    Maximizing the left-hand side over $x$ gives:
    \begin{align*}
 \max_{x \in \mathcal{S}_{i-1}^{\perp}}\frac{(x,S^\dagger HSx)}{(x,S^\dagger S x)} &\leq \tilde \lambda_i - \frac{ \lambda_m \left (\int_{\lambda < \tilde \lambda_i } (\lambda - \tilde \lambda_i)d N^\dagger P^{(H)}(\lambda) N \right ) +  \lambda_{\min}^* \left ( \delta A- \tilde \lambda_i \delta B \right ) }{\lambda_m(B) - \lambda_1(\mathcal{N}^{(B)})}
    \end{align*}
 By the min-max principle\,\cref{eq:minmax}, the left-hand side is bounded from below by $\lambda_i$. Rearranging gives the lower inequality of \cref{prop:var_principle}.\\
 
 For the upper inequality, consider the subspace $\mathcal{S}_{i} $ spanned by the leading $i$ eigenvectors of $[A,B]$. For $x \in \mathcal{S}_{i}$ arbitrary, analogous steps apply:
    \begin{align*}
 \frac{(x,SHSx)}{(x,S^\dagger S x)}  & \geq \tilde \lambda_i - \frac{(x, \left (\int_{\lambda > \tilde \lambda_i } (\lambda - \tilde \lambda_i) d N^\dagger P^{(H)}(\lambda) N + \delta A - \tilde \lambda_i \delta B \right ) x)}{(x,(B-N^\dagger N  - \delta B)x)} \\
        & \geq \tilde \lambda_i - \frac{ \lambda_1\left ( \int_{\lambda > \tilde \lambda_i } (\lambda - \tilde \lambda_i) d N^\dagger P^{(H)}(\lambda) N  \right ) + \lambda_{\max}^* \left ( \delta A- \tilde \lambda_i \delta B \right ) }{\lambda_m(B) - \lambda_1(\mathcal{N}^{(B)})}.
    \end{align*}
 The max-min principle \cref{eq:maxmin} allows us to replace the left-hand side by $\lambda_i$, and the upper inequality of \cref{prop:var_principle} follows.
\end{proof}

\begin{proof}[\textit{From \cref{prop:var_principle} to \cref{thm:master}}]
   Let $[A,B]$ be a generalizes eigenvalue problem as in \cref{prop:var_principle}. We estimate the numerator of the upper bound in Eq.\,\cref{eq:prop_upper}:
    \begin{align*}
        \lambda_1 \Bigl (\int_{\lambda > \tilde \lambda_i }  & (\lambda - \tilde \lambda_i) d N^\dagger P^{(H)}(\lambda) N \Bigr)+ \lambda_{\max}^* ( \delta A- \tilde \lambda_i \delta B  ) \\
        &\leq \Tr \left [ \int_{\lambda > \tilde \lambda_i } (\lambda - \tilde \lambda_i) d N^\dagger P^{(H)}(\lambda) N  \right ] + \lambda_{\max}^* ( \delta A - \tilde \lambda_i  \delta B ) \\
        &\leq   \int_{\lambda > \tilde \lambda_i } (\lambda - \tilde \lambda_i) d \underbrace{\Tr \left [ N^\dagger P^{(H)}(\lambda) N  \right ]}_{=\varepsilon(\lambda)} + \lambda_{\max}^* ( \delta A- \tilde \lambda_i  \delta B )
    \end{align*}
  The first inequality follows from $\lambda_1(A)\leq \Tr A$ for $A$ positive semidefinite and the second from the triangle inequality. The lower bound in Eq.\,\cref{eq:prop_lower} follows analogously. 
\end{proof}

\subsection{Related spectral inequalities}
\label{sec:related_specineq}

If $H$ is a bounded operator, we can also derive a more explicit bound from \cref{prop:var_principle} by using the operator norm.

\begin{corollary}
    \label{cor:bounded}
    Let $[A,B]$ be a generalized eigenvalue problem as in \cref{sec:descr} where the spectral subspace is a band  $\mathcal{E} = \mathcal{E}[a,b]$.
    Assume:
    \begin{enumerate}[label=\roman*)]
        \item  that $H$ is bounded by $E_{\text{min}}, E_{\text{max}} \in \mathbb{R}$ such that for all $x \in \mathcal{Q}(H)$
        \begin{align*}
         E_{\text{min}} \leq \frac{(x,Hx)}{(x,x)}\leq E_{\text{max}}.
        \end{align*}
        \item the generalized eigenvalue problem is well-conditioned as in Eq.\,\cref{eq:wellCond}
        \item  the continuous part of the spectrum of $H$ has no support in $[a,b]$
    \end{enumerate}
    Then the eigenvalues of $[A,B]$ and $H$ in $\mathcal{E}[a,b]$ have
    \begin{align}
        \frac{(E_{\text{min}} - \tilde \lambda_i) \|N_{<}  \| + \lambda_{\min}^* (\delta A - \tilde \lambda_i \delta B )  }{\lambda_m(B) - \lambda_1(\mathcal{N}^{(B)})} \leq \tilde \lambda_i -\lambda_i \leq \frac{(E_{\text{max}} - \tilde \lambda_i) \|N_{>}\| + \lambda_{\max}^* ( \delta A- \tilde \lambda_i  \delta B ) }{\lambda_m(B) - \lambda_1(\mathcal{N}^{(B)})}, \label{eq:band_ineq_bounded}
    \end{align}
    where 
    \begin{align*}
    N_{<} &= \int_{\lambda < a } d N^\dagger P^{(H)}(\lambda) N, \quad \text{and} \quad 
    N_{>} = \int_{\lambda > b }  d N^\dagger P^{(H)}(\lambda) N.
\end{align*}
\end{corollary}

We also state an alternative version of \cref{thm:master}.

\begin{theorem}
    \label{thm:other}
   Let $[A,B]$ be a generalized eigenvalue problem as in \cref{sec:descr}. The eigenvalues of $[A,B]$ and $H$ in $\mathcal{E}$ have 
    \begin{align}
 \frac{\int_{\lambda <  \lambda_i } (\lambda - \lambda_i) d \varepsilon(\lambda) + \lambda_{\min}^* (\delta A - \lambda_i \delta B )  }{\lambda_m(B) } \leq \tilde \lambda_i -\lambda_i \leq \frac{\int_{\lambda >  \lambda_i } (\lambda - \lambda_i) d \varepsilon(\lambda)+ \lambda_{\max}^* (\delta A -  \lambda_i \delta B )  }{\lambda_m(B)}. \label{eq:other}
    \end{align}
\end{theorem}

\cref{thm:other} no longer requires the spectral signal to be well-conditioned. On the other hand, the error bounds of \cref{thm:master} are given in terms of the approximate eigenvalues $\tilde \lambda_i$ rather than the unknown eigenvalues $\lambda_i$ as in \cref{thm:other}.

\begin{proof}[Proof sketch for \cref{thm:other}]
   The proof is analogous to that of \cref{prop:var_principle}, but with the roles of $[A,B]$ and $[S^\dagger H S, S^\dagger S]$ effectively exchanged. 
   For the upper bound, we use the subspace $\mathcal{S}_i$ spanned by the leading $i$ eigenvectors of $[S^\dagger H S, S^\dagger S]$ and apply the max-min principle \cref{eq:maxmin} to $[A,B]$:
   \begin{align*}
    \frac{(x,Ax)}{(x,Bx)} &= \frac{(x, ( S^\dagger H S + N^\dagger H N + \delta A)x)}{(x,Bx)} \\
    &\leq \frac{(x, ( \lambda_i S^\dagger S + N^\dagger H N + \delta A)x)}{(x,Bx)} \\
    &= \lambda_i + \frac{(x, ( N^\dagger H N - \lambda_i N^\dagger N + \delta A - \lambda_i \delta B )x)}{(x,Bx)} \\
    &= \lambda_i + \frac{(x, \left (\int (\lambda - \lambda_i) d N^\dagger P^{(H)}(\lambda) N + \delta A - \lambda_i \delta B \right ) x)}{(x,Bx)}.
   \end{align*}
   The remainder of the proof is analogous to the steps used in \cref{prop:var_principle} and \cref{thm:master}.
\end{proof}

\end{appendix}

\bibliographystyle{unsrt}
\bibliography{references}

\end{document}